\documentclass[11pt,twoside, leqno]{article}

\usepackage{amssymb}
\usepackage{amsmath}
\usepackage{mathrsfs}
\usepackage{amsthm}
\usepackage{color}

\allowdisplaybreaks

\def\ls{\lesssim}
\def\gs{\gtrsim}
\def\fz{\infty}
\renewcommand{\r}{\right}
\newcommand{\lf}{\left}

\def\ls{\lesssim}
\def\gs{\gtrsim}

\def\paz{{\partial}}

\def\rr{{\mathbb R}}

\def\rn{{{\rr}^n}}

\def\nn{{\mathbb N}}

\newcommand{\wz}{\widetilde}

\def\az{\alpha}
\def\lz{\lambda}

\def\bz{\beta}

\def\fai{\varphi}

\def\wz{\widetilde}

\def\ls{\lesssim}
\def\gs{\gtrsim}

\def\ol{\overline}

\def\boz{\Omega}

\def\esup{\mathop\mathrm{\,ess\,sup\,}}

\def\hs{\hspace{0.3cm}}

\def\dint{\displaystyle\int}
\def\dfrac{\displaystyle\frac}

\newtheorem{theorem}{Theorem}[section]
\newtheorem{lemma}[theorem]{Lemma}

\newtheorem{proposition}[theorem]{Proposition}
\theoremstyle{definition}
\newtheorem{remark}[theorem]{Remark}
\newtheorem{definition}[theorem]{Definition}

\def\loc{{\mathop\mathrm{loc}}}

\numberwithin{equation}{section}

\begin{document}

\arraycolsep=1pt

\title{\vspace{-2cm}\Large\bf
Gradient Estimates via Rearrangements for Solutions of Some Schr\"odinger Equations
\footnotetext{\hspace{-0.35cm} 2010 {\it Mathematics Subject
Classification}. {Primary 35J10; Secondary 46E30, 35J25, 35D30, 35B45.}
\endgraf{\it Key words and phrases}. Schr\"odinger equation,
Dirichlet problem, Neumann problem, gradient estimate, rearrangement, Lorentz space.
\endgraf  This project is partially supported by the National Natural Science Foundation of China
(Grant Nos.  11401276, 11571289, 11571039 and 11271175) and
the Specialized Research Fund for the Doctoral Program of Higher Education
of China (Grant No. 20120003110003).
Der-Chen Chang is partially supported by an NSF grant DMS-1408839 and a McDevitt Endowment
Fund at Georgetown University. }}
\author{Sibei Yang, Der-Chen Chang, Dachun Yang\,\footnote{Corresponding author}\,\, and Zunwei Fu }
\date{ }
\maketitle

\vspace{-0.8cm}

\begin{center}
\begin{minipage}{13.5cm}\small
{{\bf Abstract.}
In this article, by applying the well known method
for dealing with $p$-Laplace type elliptic boundary value problems,
the authors establish a sharp estimate for the decreasing rearrangement of
the gradient of solutions to the Dirichlet and the Neumann boundary value
problems of a class of Schr\"odinger equations,
under the weak regularity assumption on the boundary
of domains. As applications, gradient estimates of these solutions
in Lebesgue spaces and Lorentz spaces are obtained.}
\end{minipage}
\end{center}

\section{Introduction\label{s1}}

\hskip\parindent  It is well known that the global regularity of solutions is a classic and
interesting topic in the theory of elliptic boundary value problems. In particular,
the global estimates for the gradient of solutions to (non-)linear
elliptic boundary value problems in various function spaces
have attracted great interests for a long time; see, for example,
\cite{c92,m61,m69a,m69,m09a,m09b,s63,t76} for the linear case
and \cite{acmm10,aft00,c97,cm11,cm14a,cm14b,cm15,d83,t79} for the non-linear case.

In this article, motivated by the work in \cite{acmm10,am08,bbggpv,bg89,cm14b},
via applying the well known method for estimating the gradient
of solutions to $p$-Laplace type elliptic boundary value problems
and some estimates established by Shen \cite{sh94,sh95} for the fundamental solution of
Schr\"odinger equations,
we obtain a sharp estimate for the decreasing rearrangement of the gradient of solutions to
a class of Schr\"odinger equations with the Dirichlet or the Neumann boundary condition,
under the weak assumption for the regularity on the boundary of domains.
As applications, we further establish the gradient estimates
of solutions to these Schr\"odinger equations in Lebesgue spaces and Lorentz spaces.

To state the Schr\"odinger equations considered in this article,
we first recall the definition of the reverse H\"older class
(see, for example, \cite{g14,sh95}).
Recall that a non-negative function $w$ on $\rn$ is said to belong to the \emph{reverse
H\"older class} $RH_{q}(\mathbb{R}^n)$
with $q\in(1,\fz]$, denoted by $w\in RH_q(\rn)$, if, when
$q\in(1,\fz)$, $w\in L^q_{\loc}(\rn)$ and
\begin{align*}
[w]_{RH_{q}(\mathbb{R}^n)}:=\sup_{B\subset\rn}\lf\{\frac{1}{|B|}\int_{B}[w(x)]^q\,dx\r\}^{1/q}\lf\{
\frac{1}{|B|}\int_{B}w(x)\,dx\r\}^{-1}<\fz
\end{align*}
or, when $q=\fz$, $w\in L^\fz_{\loc}(\rn)$ and
\begin{align*}
[w]_{RH_{\fz}(\mathbb{R}^n)}:=\sup_{B\subset\rn}\lf\{\esup_{x\in B}w(x)\r\}
\lf\{\frac{1}{|B|}\int_{B}w(x)\,dx\r\}^{-1}<\fz,
\end{align*}
where the suprema are taken over all balls $B\subset\rn$.

It is well known that, for any $1<q\le p\le\fz$, $RH_p(\rn)\subset RH_q(\rn)$
(see, for example, \cite{g14}).
A typical example of the reverse H\"older class
is that, for any $x\in\rn$, $w(x):=|P(x)|^\az$,
where $P$ is a polynomial on $\rn$ and $\az\in(0,\fz)$,
which turns out to be in $RH_\fz(\rn)$
(see, for example, \cite{sh94,sh95}).
Moreover, for any $x\in\rn$, let $w(x):=|x|^{-1+\epsilon}$ with $\epsilon\in(0,1)$.
Then $w\in RH_n(\rn)$ (see, for example, \cite[Section 7]{sh95}).

Let $n\ge3$ and $\boz$ be a bounded domain in $\rn$.
Denote by $W^{1,2}(\boz)$ and $W^{1,2}_0(\boz)$ the classical \emph{Sobolev space}
on $\boz$, respectively, the closure of $C^{\fz}_c (\boz)$ in
$W^{1,2}(\boz)$, where $C^{\fz}_c (\boz)$ denotes the \emph{set of
all $C^\fz$ functions on $\rn$ with compact support contained in $\boz$}.
Assume that $0\le V\in RH_{q}(\rn)$ for some $q\in[n,\fz]$ and $V\not\equiv0$ on $\boz$.
Let $f\in L^2(\boz)$. Then a function $u\in W^{1,2}_0(\boz)$
is called a \emph{weak solution} to the Dirichlet problem
\begin{align}\label{1.1}
\lf\{\begin{array}{l}
-\Delta u+Vu=f\ \ \text{in}\ \boz,\\
u=0\ \ \ \ \ \ \ \ \ \ \ \ \ \ \ \text{on}\  \partial\boz,
\end{array}\r.
\end{align}
where $\partial\boz$ denotes the boundary of $\boz$,
if, for any $v\in W^{1,2}_0(\boz)$,
\begin{align}\label{1.2}
\int_\boz\nabla u(x)\cdot\nabla v(x)\,dx+\int_\boz V(x)u(x)v(x)\,dx=\int_\boz f(x)v(x)\,dx.
\end{align}

Assume further that $\boz$ is a bounded Lipschitz domain. Then a function $u\in W^{1,2}(\boz)$
is called a \emph{weak solution} of the Neumann problem
\begin{align}\label{1.3}
\lf\{\begin{array}{l}
-\Delta u+Vu=f\ \ \text{in}\ \boz,\\
\dfrac{\partial u}{\partial\nu}=0\ \ \ \ \ \ \ \ \ \ \ \ \ \text{on}\  \partial\boz
\end{array}\r.
\end{align}
if, for any $v\in W^{1,2}(\boz)$, \eqref{1.2} holds true.
Here and hereafter, $\nu:=(\nu_1,\,\ldots,\,\nu_n)$ denotes the \emph{outward unit normal}
to $\partial\boz$.

\begin{remark}\label{r1.1}
{\rm(i)} Let $f\in L^2(\boz)$, $0\le V\in RH_n(\rn)$ and $V\not\equiv0$ on $\boz$.
Then, by the Lax-Milgram theorem,
we know that the Dirichlet problem \eqref{1.1} and the Neumann  problem \eqref{1.3}
have the unique weak solution (see, for example, \cite[Chapter 8]{gt}).

{\rm(ii)} Assume that $V\equiv0$ in \eqref{1.1} and \eqref{1.3}, and
$\int_\boz f(x)\,dx=0$ in \eqref{1.3}.
Then, in this case, \eqref{1.1} and \eqref{1.3} are just the Dirichlet problem,
respectively, the Neumann problem of the Laplace equation in $\boz$.
\end{remark}

However, when $f\in L^1(\boz)$, the definition of the weak solution
to \eqref{1.1} or \eqref{1.3} as the way in \eqref{1.2} may be meaningless.
In this case, the generalized solutions for
\eqref{1.1} and \eqref{1.3} can be defined by an approximating method
(see Section \ref{s2} below for the details). We point out that the study for
the theory of (non-)linear elliptic boundary value problems with $L^1$-data
has aroused the attention of many mathematicians for quite some time
(see, for example, \cite{acmm10,aft00,am08,bbggpv,bg89,bs73,d96,dm97,s63}).

To state the main result of this article,
we first recall the definitions of the distribution function and the decreasing rearrangement as follows.
Let $\boz\subset\rn$ be an open bounded set and $u$ a real-valued measurable function on $\boz$.
Then the \emph{distribution function} $\mu_u:\ [0,\fz)\to[0,|\boz|]$ of $u$
is defined by setting, for any $t\in[0,\fz)$,
\begin{align*}
\mu_u(t):=|\{x\in\boz:\ |u(x)|>t\}|.
\end{align*}
The \emph{decreasing rearrangement} $u^\ast:\ [0,\fz)\to[0,\fz]$
of $u$ is defined by setting, for any $s\in[0,\fz)$,
\begin{align*}
u^\ast(s):=\sup\{t\in[0,\fz):\ \mu_u(t)>s\}.
\end{align*}
We point out that $u^\ast$ is the unique right-continuous decreasing function in $[0,\fz)$
equivalently distributed  with $u$ and, if $s\in[|\boz|,\fz)$, then $u^\ast(s)=0$,
where $|\boz|$ denotes the Lebesgue measure of $\boz$ (see, for example,
\cite[Chapter 7]{pkjf13}).

Then the main result of this article is as follows and
some necessary notions are recalled after this theorem.

\begin{theorem}\label{t1.1}
Let $n\ge3$, $\boz$ be a bounded domain in $\rn$,
$0\le V\in RH_n(\rn)$ and $V\not\equiv0$ on $\boz$.
Assume that $\partial\boz\in W^2L^{n-1,1}$ or $\boz$ is semi-convex,
$f\in L^1(\boz)$ and $u$ is the unique weak solution to the Dirichlet
problem \eqref{1.1} or the Neumann problem \eqref{1.3}.
Then there exists a positive constant $C$, depending on $n$, $[V]_{RH_n(\rn)}$
and $\boz$, such that, for any $s\in(0,|\boz|)$,
\begin{align}\label{1.4}
|\nabla u|^\ast(s)\le C\lf[s^{-\frac1{n'}}\int_0^sf^{\ast}(r)\,dr
+\int_s^{|\boz|}f^{\ast}(r)r^{-\frac1{n'}}\,dr\r],
\end{align}
where $n':=n/(n-1)$, $|\nabla u|^\ast$ and $f^\ast$ denote the
decreasing rearrangements of $|\nabla u|$, respectively, $f$.
\end{theorem}

\begin{remark}\label{r1.2}
{\rm(i)} We point out that the estimate \eqref{1.4} is sharp
in some sense. Here we give an example to explain this.

Let $n:=3$ and $B_0:=B(0,1)$ be the ball in $\rr^3$ with the center $0$ and the radius $1$.
For any $x:=(x_1,x_2,x_3)\in B_0$, let
$u_0(x):=|x|-1$ and
\begin{align*}
f_0(x):=\begin{cases}
-\dfrac{2}{|x|}+|x|-1\  &\text{when}\ x\neq0,\\
0\  \ \quad\quad\quad\quad\quad &\text{when}\ x=0.
\end{cases}
\end{align*}
Then, by a simple calculus, we find that
$u_0\in W^{1,2}_0(B_0)$ has the weak gradient that,
for any $x\in B_0$,
\begin{align*}
\nabla u_0(x)=\begin{cases}
\lf(\dfrac{x_1}{|x|}, \dfrac{x_2}{|x|}, \dfrac{x_3}{|x|}\r)\  &\text{when}\ x\neq0,\\
0\  \ \ \ \ \ \ \ \ \ \ \ \  \ &\text{when}\ x=0,
\end{cases}
\end{align*}
and $f_0\in L^2(B_0)$. Moreover, it is easy to see that,
for any $x\in B_0\setminus\{0\}$, $|\nabla u_0(x)|=1$,
which, together with the definition of $|\nabla u_0|^\ast$, implies that,
for any $s\in(0,|B_0|)$, $|\nabla u_0|^\ast(s)=1$.
Furthermore, from the definitions of $u_0$ and $f_0$,
it follows that $u_0$ is the weak solution of
the Dirichlet problem
\begin{align*}
\begin{cases}
-\Delta u+u=f_0\ \ &\text{in}\ B_0,\\
u=0\ \ \ \ \ \ \quad\quad\ &\text{on}\  \partial B_0.
\end{cases}
\end{align*}
Using calculus, we know that, for any $t\in(0,\fz)$,
\begin{align*}
\mu_{f_0}(t)=\begin{cases}
|B_0|\  \hspace{4cm} &\text{when}\ t\in(0,2],\\
\lf(\dfrac{1-t+[(t-1)^2+8]^{1/2}}{2}\r)^{3}|B_0|\  \ \ \ \ \ &\text{when}\ t\in(2,\fz),
\end{cases}
\end{align*}
which further implies that, for any $s\in(0,|B_0|)$,
$$f_0^\ast(s)=2\lf(\frac{s}{|B_0|}\r)^{-\frac13}-\lf(\frac{s}{|B_0|}\r)^{\frac13}+1.
$$
From this and $n=3$, it follows that, for any $s\in(0,|B_0|)$,
\begin{align*}
&s^{-\frac{1}{n'}}\int_0^s f_0^\ast(r)\,dr+\int_s^{|B_0|}f_0^\ast(r)r^{-\frac{1}{n'}}\,dr\\
&\hs=\frac{9}{2}|B_0|^{\frac13}+2|B_0|^{\frac13}\log\lf(\frac {|B_0|}s\r)+\frac{3}{4}
|B_0|^{-\frac13}s^{\frac23}-2s^{\frac13},
\end{align*}
which, combined with \eqref{1.4} and the fact that, for any $s\in(0,|B_0|)$, $|\nabla u_0|^\ast(s)=1$,
further implies that there exists a positive constant $C$, depending on $|B_0|$, such that,
for any  $s\in(|B_0|/2,|B_0|)$,
$$C^{-1}|\nabla u_0|^\ast(s)\le
s^{-\frac{1}{n'}}\int_0^s f_0^\ast(r)\,dr+\int_s^{|B_0|}f_0^\ast(r)r^{-\frac{1}{n'}}\,dr
\le C|\nabla u_0|^\ast(s).
$$
Thus, the estimate \eqref{1.4} is sharp in some sense.

{\rm(ii)} We point out that we obtain \eqref{1.4} in Theorem \ref{t1.1} inspired
by the proof of \cite[Theorem 1.1]{cm14b}.

{\rm(iii)} Assume that $\boz$, $u$ and $f$ are as in Theorem \ref{t1.1}.
We point out that the right-hand side of \eqref{1.4} can be naturally expressed via
the Calder\'on operator. Indeed,
let $1\le p_0<p_1\le\fz$, $q_0,\,q_1\in[1,\fz]$, $q_0\neq q_1$
and $\az:=(1/q_0-1/q_1)/(1/p_0-1/p_1)$.
Recall that the \emph{Calder\'on operator} $S^{p_0,\,p_1}_{q_0,\,q_1}$ is defined by setting,
for any measurable function $g$ on $(0,\fz)$ and $t\in(0,\fz)$,
$$S^{p_0,\,p_1}_{q_0,\,q_1}(g)(t):=t^{-\frac{1}{q_0}}\int_0^{t^\az}s^{\frac1{p_0}}g(s)\,\frac{ds}{s}
+t^{-\frac{1}{q_1}}\int_{t^\az}^{\fz}s^{\frac1{p_1}}g(s)\,\frac{ds}{s},
$$
which plays a key role in the theory of rearrangement invariant spaces (see, for example,
\cite{bs88}). By \eqref{1.4} and the fact that, for any $s\in[|\boz|,\fz)$,
$|\nabla u|^\ast(s)=0=f^\ast(s)$, we find that, for any $s\in(0,\fz)$,
$$|\nabla u|^\ast(s)\le CS^{1,\,n}_{n',\,\fz}(f^\ast)(s)
$$
with $C$ same as in \eqref{1.4}.

Moreover, for any $s\in(0,\fz)$, let $f^{\ast\ast}(s):=(1/s)\int_0^s f^\ast(t)\,dt$.
Then, by using \eqref{1.4}, we can also obtain another control of $|\nabla u|^\ast$
via $f^{\ast\ast}$. Indeed, it is easy to see that, for any $s\in(0,\fz)$, $f^\ast(s)\le f^{\ast\ast}(s)$ and
\begin{align*}
\frac{d}{ds}f^{\ast\ast}(s)=\frac{1}{s}[f^\ast(s)-f^{\ast\ast}(s)],
\end{align*}
which, together with \eqref{1.4}, integration by parts and the fact that
$\|f\|_{L^1(\boz)}=\int_{0}^{|\boz|}f^\ast(r)\,dr$,
implies that, for any $s\in(0,|\boz|)$,
\begin{align*}
|\nabla u|^\ast(s)&\le C\lf\{s^{\frac1{n}}f^{\ast\ast}(s)+\int_s^{|\boz|}
\lf[r\frac{d}{dr}f^{\ast\ast}(r)+f^{\ast\ast}(r)\r]r^{-\frac1{n'}}\,dr\r\}\\
&\le C\lf\{s^{\frac1{n}}f^{\ast\ast}(s)+f^{\ast\ast}(r)r^{\frac 1n}\Big|^{|\boz|}_s
+\frac{1}{n'}\int_s^{|\boz|}f^{\ast\ast}(r)r^{-\frac1{n'}}\,dr\r\}\\
&\le C\lf[|\boz|^{-\frac{1}{n'}}\|f\|_{L^1(\boz)}+\frac{1}{n'}
\int_s^{|\boz|}f^{\ast\ast}(r)r^{-\frac1{n'}}\,dr\r],
\end{align*}
where $C$ is same as in \eqref{1.4}.

{\rm(iv)} As was pointed out in \cite[p.\,573]{cm14b}, the particular
feature of \eqref{1.4} is its independence of concrete function spaces.
Thus, it is agile enough to obtain the estimates of $|\nabla u|$
in various rearrangement invariant function spaces via \eqref{1.4}
(see Theorem \ref{t1.2} below).

{\rm(v)} $\partial\boz\in W^2L^{n-1,1}$ means that $\boz$ is locally
a subgraph of a function of $n-1$ variables whose second order
derivatives belong to the Lorentz space $L^{n-1,1}$. It is worth pointing out that
$\partial\boz\in W^2L^{n-1,1}$ is the weakest possible integrability condition
on second order derivatives guaranteeing the first order derivatives to be continuous, which further implies
$\partial\boz\in C^{1}$ (see, for example, \cite{cp98}).
\end{remark}

We prove Theorem \ref{t1.1} by following the method in \cite{cm14b}
to estimate the rearrangement of the gradient of solutions to $p$-Laplace
type elliptic boundary value problems.
More precisely, by the boundedness of the gradient of solutions to \eqref{1.1} and \eqref{1.3}
with $f\in L^{n,1}(\boz)$, obtained in \cite{y15} (see also Proposition \ref{p3.1} below),
the $L^{n/(n-1),\fz}(\boz)$-estimates for the gradient of solutions to
\eqref{1.1} and \eqref{1.3} with $f\in L^1(\boz)$ (see Proposition \ref{p3.2} below) and the method of
the $K$-functional, we show Theorem \ref{t1.1}. Here and hereafter, 
$L^{p,q}(\boz)$, with $p\in(0,\fz]$ and $q\in(0,\fz]$,
denotes the Lorentz space on $\boz$.

Recall that, for the Dirichlet problem or the Neumann problem of the Laplace equation,
the following conclusion was obtained in \cite[Theorems 1.1 and 1.2]{cm14b}
(see also \cite[Theorem 3.2]{cm15}).

\setcounter{theorem}{0}
\renewcommand{\thetheorem}{\arabic{section}.\Alph{theorem}}

\begin{theorem}\label{thm-a}
Let $n\ge3$, $\boz$ be a bounded domain in $\rn$ and $f\in L^1(\boz)$.
Assume that $\partial\boz\in W^2L^{n-1,1}$ or $\boz$ is convex,
and $u$ is the unique weak solution to the Dirichlet problem
or the Neumann problem of the Laplace equation in $\boz$.
Then there exists a positive constant $C$, depending on $n$
and $\boz$, such that, for any $s\in(0,|\boz|)$,
\begin{align}\label{1.5}
|\nabla u|^\ast(s)\le C\int_s^{|\boz|}f^{\ast\ast}(r)r^{-\frac1{n'}}\,dr.
\end{align}
\end{theorem}

\setcounter{theorem}{3}
\renewcommand{\thetheorem}{\arabic{section}.\arabic{theorem}}

\begin{remark}\label{r1.3}
{\rm(i)}
We point out that \eqref{1.5} may not be true
for some boundary value problems of the Laplace equation.
Indeed, there exists the following counterexample illustrating this.

Let $n:=3$, $B_0$ and $u_0$ be as in Remark \ref{r1.2}(i).
For any $x\in B_0$, let
\begin{align*}
f_1(x):=\begin{cases}
-\dfrac{2}{|x|}\quad  &\text{when}\ x\neq0,\\
0\  \ \ \ \ &\text{when}\ x=0.
\end{cases}
\end{align*}
Then $f_1\in L^2(B_0)$ and $u_0$ is the weak solution of
the Dirichlet problem
\begin{align*}
\lf\{\begin{array}{l}
-\Delta u=f_1\ \ \text{in}\ B_0,\\
u=0\ \ \ \ \ \ \ \ \text{on}\  \partial B_0.
\end{array}\r.
\end{align*}
Moreover, from the definition of $f_1$, $n=3$ and a simple calculus, it follows that,
for any $s\in(0,|B_0|)$, $f_1^{\ast\ast}(s)=3|B_0|^{1/3}s^{-1/3}$ and hence
\begin{align*}
\int_s^{|B_0|}f_1^{\ast\ast}(r)r^{-\frac1{n'}}\,dr=3|B_0|^{\frac13}
\int_{s}^{|B_0|}r^{-1}\,dr
=3|B_0|^{\frac13}\log\lf(\frac{|B_0|}{s}\r).
\end{align*}
By this, $\lim_{s\to|B_0|}\log(|B_0|/s)=0$ and the fact that, for any $s\in(0,|B_0|)$,
$|\nabla u_0|^\ast(s)=1$, we conclude that \eqref{1.5} does not hold true in this case.

{\rm(ii)} Let $\boz$, $f$ and $u$ be as in Theorem \ref{thm-a}.
Indeed, for the Dirichlet problem and the Neumann problem of the Laplace equation in $\boz$,
the same estimate as in \eqref{1.4} for $|\nabla u|^\ast$ was established
in the proof of \cite[Theorems 1.1 and 1.2]{cm14b}.
\end{remark}

Now we recall the definitions
of the semi-convex domain in $\rn$ and the Lorentz(-Sobolev) space.

\begin{definition}\label{d1.1}
(i) Let $\boz$ be an open set in $\rn$. The collection of
\emph{semi-convex functions} on $\boz$ consists of continuous functions
$u:\ \boz\rightarrow\rr$ with the property that there exists a positive
constant $\widetilde{C}$ such that, for all $x,\,h\in\rn$ with the ball
$B(x,|h|)\subset \boz$,
$$2u(x)-u(x+h)-u(x-h)\le \widetilde{C}|h|^2.
$$
The best constant $\widetilde{C}$ above is referred as the
\emph{semi-convexity constant} of $u$.

(ii) A non-empty, proper open subset $\boz$ of $\rn$ is said to be
\emph{semi-convex} if there exist $b,\,c\in(0,\fz)$ such
that, for every $x_0\in\paz\boz$, there exist an
$(n-1)$-dimensional affine variety $H\subset\rn$ passing through
$x_0$, a choice $N$ of the unit normal to $H$, and an open set
$$\mathcal{C}:=\{\wz{x}+tN:\
\wz{x}\in H,\ |\wz{x}-x_0|<b,\ |t|<c\}$$
(called a
\emph{coordinate cylinder} near $x_0$ with axis along $N$) satisfying,
for some semi-convex function $\fai:\ H\rightarrow\rr$,
$$\mathcal{C}\cap\boz=\mathcal{C}\cap\{\wz{x}+tN:\ \wz{x}\in H,\
t>\fai(\wz{x})\},$$
$$\mathcal{C}\cap\paz\boz=\mathcal{C}\cap\{\wz{x}+tN:\ \wz{x}\in H,\
t=\fai(\wz{x})\},$$
$$\mathcal{C}\cap\ol{\boz}
^{\complement}=\mathcal{C}\cap\{\wz{x}+tN:\ \wz{x}\in H,\
t<\fai(\wz{x})\},$$
$$\fai(x_0)=0\ \text{and}\ |\fai(\wz{x})|<c/2\ \text{when}\
|\wz{x}-x_0|\le b,$$
where $\ol{\boz}$ and $\ol{\boz}^{\complement}$
denote the closure of $\boz$ in $\rn$, respectively, the complementary set of $\overline{\boz}$ in $\rn$.
\end{definition}

\begin{remark}\label{r1.4}
It is well known that bounded semi-convex domains
in $\rn$ are bounded Lipschitz domains, and convex
domains in $\rn$ are semi-convex domains
(see, for example, \cite{mmmy10,mmy10}).
\end{remark}

Now we recall the definitions of Lorentz spaces and Lorentz-Sobolev spaces.
Let $q\in[1,\fz]$, $s\in(0,\fz]$ and $\boz$ be a bounded domain in $\rn$.
Then the \emph{Lorentz space} $L^{q,s}(\boz)$ is defined to be the set of all measurable
functions $u:\ \boz\to\rr$ satisfying
\begin{align*}
\|u\|_{L^{q,s}(\boz)}:=\begin{cases}
\lf[\dint_0^{|\boz|}\lf\{t^{\frac{1}{q}-\frac{1}{s}}
u^\ast(t)\r\}^{s}\,dt\r]^{\frac1s}<\fz \ \ &\text{when} \ s\in(0,\fz),\\
\sup\limits_{t\in(0,|\boz|]}\lf[t^{\frac1q}u^\ast(t)\r]<\fz\ \ \ \ \ \ \ \ \
\ \ \ \ &\text{when}\  s=\fz.
\end{cases}
\end{align*}

Moreover, we recall some necessary conclusions for Lorentz spaces as follows.

\begin{remark}\label{r1.5}
For Lorentz spaces, the following facts hold true (see, for example, \cite[Chapter 8]{pkjf13}):
\begin{itemize}
\item[{\rm(i)}] For $q\in[1,\fz]$, $L^{q,q}(\boz)=L^q(\boz)$.
\item [{\rm(ii)}] If $q\in[1,\fz]$ and $s_1,s_2\in(0,\fz]$ with $s_1<s_2$,
then $L^{q,s_1}(\boz)\subsetneqq L^{q,s_2}(\boz)$.
\item [{\rm(iii)}] If $q_1,\,q_2\in[1,\fz]$ with $q_1>q_2$ and $s_1,\,s_2\in(0,\fz]$,
then $L^{q_1,s_1}(\boz)\subsetneqq L^{q_2,s_2}(\boz)$.
\end{itemize}
\end{remark}

Let $m\in\nn$, $q\in[1,\fz)$ and $s\in[1,\fz]$.
The \emph{Lorentz-Sobolev space} $W^mL^{q,s}(\boz)$ is defined by
\begin{align}\label{1.6}
W^mL^{q,s}(\boz)&:=\{u\in L^{q,s}(\boz):\ u\ \text{is}\
m\text{-times weakly differentiable in} \\ \nonumber
&\hs\hspace{3.8cm}\boz\ \text{and}\ |\nabla^ku|\in L^{q,s}(\boz),\ 1\le k\le m\}
\end{align}
equipped with the norm
$$\|u\|_{W^mL^{q,s}(\boz)}:=\|u\|_{L^{q,s}(\boz)}+\sum_{k=1}^m
\||\nabla^ku|\|_{L^{q,s}(\boz)}.
$$

Furthermore, it is worth pointing out that Lorentz spaces and Sobolev-Lorentz spaces extend
Lebesgue spaces, respectively, Sobolev spaces. Furthermore, the Lorentz-Zygmund space
is a further extension of the Lorentz space.
Recall that, for any $q\in(1,\fz]$, $k\in(0,\fz]$ and $\bz\in\rr$ or
$q=1$, $k\in(0,1]$ and $\bz\in[0,\fz)$, the \emph{Lorentz-Zygmund space}
$L^{q,k}(\log L)^{\bz}(\boz)$ is defined as the set of
all measurable functions $u$ on $\boz$ satisfying that
$$\|u\|_{L^{q,k}(\log L)^{\bz}(\boz)}:=\lf[\int_0^{|\boz|}\lf\{s^{\frac1q-\frac1k}
\lf[1+\log\lf(\frac{|\boz|}{s}\r)\r]^\bz u^\ast(s)\r\}^k\,ds\r]^{\frac1k}<\fz
$$
(see, for example, \cite[Chapter 9]{pkjf13} and \cite[p.\,588]{cm14b}
for some details about Lorentz-Zygmund spaces).

As applications of Theorem \ref{t1.1}, we have the following
estimates for the gradient of solutions to \eqref{1.1} and \eqref{1.3}
in the scales of Lebesgue spaces and Lorentz(-Zygmund) spaces. We also point out that
the estimates for the gradient of solutions to the Dirichlet or the Neumann
boundary value problems of some (non-)linear elliptic equations in
Lebesgue spaces and Lorentz spaces were studied in
\cite{acmm10,aft00,cm11,cm14a,cm14b,m09a,t76,t79}.

\begin{theorem}\label{t1.2}
Let $n$, $\boz$, $V$ and $u$ be as in Theorem \ref{t1.1}.
Assume that $f\in L^{q,k}(\boz)$ with $q\in[1,\fz]$ and $k\in(0,\fz]$.
\begin{itemize}
  \item[{\rm(i)}] If $q=1=k$ (in this case, $L^{q,k}(\boz)=L^1(\boz)$),
  then, for any $p\in[1,n/(n-1))$, there exists a positive constant $C$,
  depending on $n$, $p$, $[V]_{RH_n(\rn)}$ and $\boz$, such that
  $$\|\nabla u\|_{L^p(\boz)}\le C\|f\|_{L^1(\boz)}.
  $$

 \item[{\rm(ii)}] If $q=1$ and $k\in(0,1]$, then there exists a positive constant $C$,
  depending on $n$, $k$, $[V]_{RH_n(\rn)}$ and $\boz$, such that
  $$\|\nabla u\|_{L^{\frac{n}{n-1},\fz}(\boz)}\le C\|f\|_{L^{1,k}(\boz)}.
  $$

  \item[{\rm(iii)}] If $q\in(1,n)$ and $k=q$ (in this case, $L^{q,k}(\boz)=L^q(\boz)$),
  then there exists a positive constant $C$,
  depending on $n$, $q$, $[V]_{RH_n(\rn)}$ and $\boz$, such that
  $$\|\nabla u\|_{L^{\frac{nq}{n-q}}(\boz)}\le C\|f\|_{L^q(\boz)}.
  $$
 \item[{\rm(iv)}] If $q\in(1,n)$ and $k\in(0,\fz]$, then there exists a positive constant $C$,
  depending on $n$, $q$, $k$, $[V]_{RH_n(\rn)}$ and $\boz$, such that
  $$\|\nabla u\|_{L^{\frac{nq}{n-q},k}(\boz)}\le C\|f\|_{L^{q,k}(\boz)}.
  $$

  \item[{\rm(v)}] If $q=n=k$ (in this case, $L^{q,k}(\boz)=L^n(\boz)$),
  then, for any $p\in[1,\fz)$, there exists a positive constant $C$,
  depending on $n$, $p$, $[V]_{RH_n(\rn)}$ and $\boz$, such that
  $$\|\nabla u\|_{L^{p}(\boz)}\le C\|f\|_{L^n(\boz)}.
  $$
   \item[{\rm(vi)}] If $q=n$ and $k\in(1,\fz]$, then there exists a positive constant $C$,
  depending on $n$, $k$, $[V]_{RH_n(\rn)}$ and $\boz$, such that
  $$\|\nabla u\|_{L^{\fz,k}(\log L)^{-1}(\boz)}\le C\|f\|_{L^{n,k}(\boz)}.
  $$
  \item[{\rm(vii)}] If $q\in(n,\fz]$ and $k=q$ (in this case, $L^{q,k}=L^q(\boz)$),
  then there exists a positive constant $C$,
  depending on $n$, $q$, $[V]_{RH_n(\rn)}$ and $\boz$, such that
  $$\|\nabla u\|_{L^{\fz}(\boz)}\le C\|f\|_{L^q(\boz)}.
  $$
  \item[{\rm(viii)}] If either $q=n$ and $k\in(0,1]$ or $q\in(n,\fz]$ and $k\in(0,\fz]$,
  then there exists a positive constant $C$,
  depending on $n$, $q$, $k$, $[V]_{RH_n(\rn)}$ and $\boz$, such that
  $$\|\nabla u\|_{L^{\fz}(\boz)}\le C\|f\|_{L^{q,k}(\boz)}.
  $$
\end{itemize}
\end{theorem}

To prove Theorem \ref{t1.2}, motivated by \cite[Corollary 4.1]{cm14b}, via using Theorem \ref{t1.1},
we establish a general criterion for
the estimates of the gradient of solutions to \eqref{1.1} and \eqref{1.3} in
the rearrangement invariant quasi-normed function space (see Lemma \ref{l4.1} below).
Then, using this criterion, several weighted type Hardy inequalities obtained in
\cite{cs93,m11} and the fact that Lebesgue spaces and Lorentz(-Zygmund) spaces are rearrangement
invariant quasi-normed spaces, we show Theorem \ref{t1.2}.

Moreover, we point out that, applying the method for proving Theorem \ref{t1.2},
the estimates for the gradient of solutions to \eqref{1.1} and \eqref{1.3}
in some Orlicz spaces and Lorentz-Zygmund spaces can be obtained.
The details are omitted here and we refer the readers to \cite{c96,c00,cm14b}
for the related work of the estimates for the gradient of solutions to
$p$-Laplace type elliptic boundary value problem in Orlicz spaces
and Lorentz-Zygmund spaces.

The layout of this article is as follows. In Section \ref{s2},
we recall the generalized weak solutions for the Dirichlet problem \eqref{1.1}
and the Neumann problem \eqref{1.3} with $L^1(\boz)$-data;
and then we give out the proofs of Theorems \ref{t1.1} and \ref{t1.2} in
Sections \ref{s3} and \ref{s4}, respectively.

Finally we make some conventions on notation. Throughout the whole
article, we always denote by $C$ a \emph{positive constant} which is
independent of the main parameters, but it may vary from line to
line. The \emph{symbol} $A\ls B$ means that $A\le CB$. If $A\ls
B$ and $B\ls A$, then we write $A\sim B$.
For any measurable subset $E$ of $\rn$, we denote by $\chi_{E}$ its
\emph{characteristic function}. We also let $\nn:=\{1,\, 2,\, \ldots\}$.
Moreover, for $q\in[1,\fz]$, we denote by $q'$
its \emph{conjugate exponent}, namely, $1/q + 1/q'= 1$.

\section{Solutions with $L^1(\boz)$-data\label{s2}}

\hskip\parindent In this section, we recall the definition of the
generalized weak solutions to the Dirichlet problem \eqref{1.1}
and the Neumann problem \eqref{1.3} with $L^1$-data.

By Remark \ref{r1.1}, we know that, when $f\in L^2(\boz)$,
the Dirichlet problem \eqref{1.1} and the Neumann problem \eqref{1.3}
have the unique weak solutions $u\in W^{1,2}_0(\boz)$, respectively, $u\in W^{1,2}(\boz)$.

However, when $f\in L^1(\boz)$, $f$ may not be in
the dual space of $W^{1,2}_0(\boz)$. In this case, the definition of the weak solutions
to \eqref{1.1} and \eqref{1.3} as the way in \eqref{1.2} may be meaningless.
Thus, when $f\in L^1(\boz)$, we need another way to define the weak solution
for the Dirichlet problem \eqref{1.1} and the Neumann problem \eqref{1.3}. We also point out that
(non-)linear elliptic problems with $L^1$-data have attracted great interests
for a long time, since they play important roles in partial differential equations
(see, for example, \cite{aft00,am08,bbggpv,bg89,bs73,cm14b,d96,dm97,s63}
and the references therein).

Now we recall the definitions of approximable solutions to the Dirichlet
problem \eqref{1.1} and the Neumann problem \eqref{1.3}, which are based on
the way of sequences of weak solutions as in \eqref{1.2} to approximate
(see, for example, \cite{d96,dm97}, for the related notion of the approximable solution).

Let $W^{1,1}(\boz)$ be the Sobolev space  defined as in \eqref{1.6} with $m=1$ and
$q=s=1$. Denote by $W^{1,1}_0(\boz)$ the closure of $C^{\fz}_c (\boz)$ in $W^{1,1}(\boz)$.

\begin{definition}\label{d2.1}
Let $n\ge3$ and $\boz$ be a bounded domain in $\rn$. Assume that $V$ is as in \eqref{1.1}
and $f\in L^1(\boz)$.

{\rm(i)} A function $u\in W^{1,1}_0(\boz)$ is called an \emph{approximable solution}
to the Dirichlet problem \eqref{1.1} if there exists a sequence $\{f_k\}_{k\in\nn}\subset
L^2(\boz)$ such that $f_k\to f$ in $L^1(\boz)$
as $k\to\fz$ and the sequence of weak solutions, $\{u_k\}_{k\in\nn}\subset W^{1,2}_0(\boz)$,
to \eqref{1.1}, with $f$ replaced by $f_k$, satisfies that
$u_k\to u$ almost everywhere in $\boz$
as $k\to\fz$.

{\rm(ii)} Assume further that $\boz$ is a Lipschitz domain. Then a function
$u\in W^{1,1}(\boz)$ is called an \emph{approximable solution}
to the Neumann problem \eqref{1.3} if there exists a sequence $\{f_k\}_{k\in\nn}\subset
L^2(\boz)$ such that $f_k\to f$ in $L^1(\boz)$
as $k\to\fz$ and the sequence of weak solutions, $\{u_k\}_{k\in\nn}\subset W^{1,2}(\boz)$,
to \eqref{1.3}, with $f$ replaced by $f_k$, satisfies that
$u_k\to u$ almost everywhere in $\boz$ as $k\to\fz$.
\end{definition}

Then the existence and the uniqueness for approximate solutions
are as follows.

\begin{proposition}\label{p2.1}
Let $n\ge3$ and $\boz$ be a bounded domain in $\rn$. Assume that $V$ is as in \eqref{1.1}
and $f\in L^1(\boz)$.

{\rm(i)} Then there exists a unique approximable solution $u\in W^{1,1}_0(\boz)$
to \eqref{1.1} satisfying that, for any $v\in C^\fz_c(\boz)$,
\begin{align}\label{2.1}
\int_\boz \nabla u(x)\cdot\nabla v(x)\,dx+\int_\boz V(x)u(x)v(x)\,dx=\int_\boz f(x)v(x)\,dx.
\end{align}
Moreover, if $\{u_k\}_{k\in\nn}$ is a sequence of approximating solutions for $u$,
then there exists a subsequence of $\{u_k\}_{k\in\nn}$, still denoted by $\{u_k\}_{k\in\nn}$,
such that $\nabla u_k\to\nabla u$ almost everywhere in $\boz$ as $k\to\fz$.

{\rm(ii)} Assume further that $\boz$ is a Lipschitz domain.
Then there exists a unique approximable solution $u\in W^{1,1}(\boz)$
to \eqref{1.3} such that \eqref{2.1} holds true for any $v\in C^\fz(\boz)$.
Moreover, if $\{u_k\}_{k\in\nn}$ is a sequence of approximating solutions for $u$,
then there exists a subsequence of $\{u_k\}_{k\in\nn}$, still denoted by $\{u_k\}_{k\in\nn}$,
such that $\nabla u_k\to\nabla u$ almost everywhere in $\boz$ as $k\to\fz$.
\end{proposition}

\begin{proof}
The conclusion of (i) was essentially obtained in \cite{bg89} and the proof of (ii)
is similar, the details being omitted here.
\end{proof}

\section{Proof of Theorem \ref{t1.1}\label{s3}}

\hskip\parindent In this section, we give out the proof of Theorem \ref{t1.1}.
We begin with the following auxiliary conclusion, which was obtained in \cite{y15}.

\begin{proposition}\label{p3.1}
Let $n\ge3$, $\boz$ be a bounded domain in $\rn$,
$0\le V\in RH_n(\rn)$ and $V\not\equiv0$ on $\boz$.
Assume that $\partial\boz\in W^2L^{n-1,1}$ or $\boz$ is semi-convex, $f\in L^{n,1}(\boz)$
and $u$ is the unique weak solution to the Dirichlet
problem \eqref{1.1} or the Neumann problem \eqref{1.3}.
Then there exists a positive constant $C$, depending on $n$, $[V]_{RH_n(\rn)}$ and $\boz$, such that
\begin{align*}
\|\nabla u\|_{L^\fz(\boz)}\le C\|f\|_{L^{n,1}(\boz)}.
\end{align*}
\end{proposition}

\begin{proposition}\label{p3.2}
Let $n\ge3$, $\boz$ be a bounded Lipschitz domain in $\rn$,
$0\le V\in RH_n(\rn)$ and $V\not\equiv0$ on $\boz$.
Assume that  $f,\,g\in L^1(\boz)$. Let $u$ be the unique approximable solution to \eqref{1.1} or \eqref{1.3},
and $v$ the solution to the same problem with $f$ replaced by $g$. Then there exists a positive
constant $C$, depending on $n$, $[V]_{RH_n(\rn)}$ and $\boz$, such that
\begin{align}\label{3.1}
\|\nabla u-\nabla v\|_{L^{n',\fz}(\boz)}\le C\|f-g\|_{L^1(\boz)}.
\end{align}
\end{proposition}

To prove Proposition \ref{p3.2}, we need a differential inequality for the distribution
function of Sobolev functions established by Maz'ya \cite{m69}.
We first recall the following notions. Let $\boz\subset\rn$
be a bounded domain. For any measurable function $u$ on $\boz$,
denote by $u_+$ and $u_-$ the positive part, respectively, the negative part of $u$, namely,
$u_+:=(|u|+u)/2$ and $u_-:=(|u|-u)/2$. Moreover, the \emph{median}
$\mathrm{med}(u)$ of $u$ is defined by setting
$$\mathrm{med}(u):=\sup\{t\in\rr:\ |\{x\in\boz:\ u(x)\ge t\}|\ge|\boz|/2\}.
$$

The following lemma was obtained by Maz'ya in \cite[Lemma 2]{m69}.

\begin{lemma}\label{l3.1}
Let $\boz\subset\rn$ be a bounded domain.
\begin{itemize}
  \item[{\rm(i)}] Then there exists a positive constant $C$, depending on $n$,
  such that, for any $u\in W^{1,2}_0(\boz)$
  and almost every $t\in[0,\fz)$,
  $$1\le C[-\mu'_u(t)]^{\frac1{2}}[\mu_u(t)]^{-\frac1{n'}}\lf\{-\frac{d}{dt}\int_{\{|u|>t\}}|\nabla u(x)|^2\,dx\r\}^{\frac12},
  $$
  here and hereafter, $\mu_u$ denotes the distribution function of $u$, $\mu'_u$ denotes the derivative of $\mu_u$
  and $\{|u|>t\}:=\{x\in\boz:\ |u(x)|>t\}$.
  \item[{\rm(ii)}] Assume further that $\boz$ is a Lipschitz domain. Then there exists a positive constant
  $C$, depending on $n$ and $\boz$, such that, for any $u\in W^{1,2}(\boz)$ with $\mathrm{med}(u)=0$ and almost every
  $t\in[0,\fz)$,
  $$1\le C[-\mu'_{u_\pm}(t)]^{\frac1{2}}[\mu_{u_\pm}(t)]^{-\frac1{n'}}
  \lf\{-\frac{d}{dt}\int_{\{u_\pm>t\}}|\nabla u(x)|^2\,dx\r\}^{\frac12}.
  $$
\end{itemize}
\end{lemma}

Furthermore, to prove Proposition \ref{p3.2}, we need the following conclusion.

\begin{lemma}\label{l3.2}
Let $n\ge3$, $\boz\subset\rn$ be a bounded Lipschitz domain and $f,\,g\in L^1(\boz)$.
Assume that $V$ is as in Theorem \ref{t1.1}.
\begin{itemize}
  \item[{\rm(i)}] Let $u$ be the unique approximable
  solution to \eqref{1.1} and $v$ the solution to \eqref{1.1} with $f$ replaced by $g$.
  Then there exists a positive constant $C$, depending on $n$, $[V]_{RH_n(\rn)}$ and $\boz$, such that
  \begin{align}\label{3.2}
  \|u-v\|_{L^{\frac{n}{n-2},\fz}(\boz)}\le C\|f-g\|_{L^1(\boz)}.
  \end{align}
  \item[{\rm(ii)}] Let $u$ be the unique approximable
  solution to \eqref{1.3} and $v$ the solution to \eqref{1.3} with $f$ replaced by $g$.
  Then there exists a positive constant $C$, depending on $n$, $[V]_{RH_n(\rn)}$ and $\boz$, such that
  \begin{align}\label{3.3}
  \lf\|(u-v)-\mathrm{med}(u-v)\r\|_{L^{\frac{n}{n-2},\fz}(\boz)}\le C\|f-g\|_{L^1(\boz)}.
   \end{align}
\end{itemize}
\end{lemma}

To show Lemma \ref{l3.2}, we need the following conclusion,
which was established in \cite{y15} via using some estimates obtained
by Shen \cite{sh94,sh95} for the fundamental solution of Schr\"odinger equations.
The details are omitted here.

\begin{lemma}\label{l3.3}
Let $n\ge3$, $\boz\subset\rn$ be a bounded Lipschitz domain and $0\le V\in RH_n(\rn)$.
Denote by $L_\boz:=-\Delta+V$ the Schr\"odinger operator on $\boz$ with
the Dirichlet or the Neumann boundary condition.
Then the operator $VL^{-1}_\boz$ is bounded on $L^p(\boz)$ with $p\in[1,n]$.
\end{lemma}

Now we prove Lemma \ref{l3.2} by using Lemmas \ref{l3.1} and \ref{l3.3}.

\begin{proof}[Proof of Lemma \ref{l3.2}]
Here we only show \eqref{3.3}, since the proof of \eqref{3.2} is similar, the details being omitted.
We first assume that $f,\,g\in L^2(\boz)$, $u$ is the weak solution to \eqref{1.3} and
$v$ the weak solution to \eqref{1.3} with $f$ replaced by $g$. Let
\begin{align}\label{3.4}
w:=(u-v)-\mathrm{med}(u-v).
\end{align}
Moreover, for any given $t,\,h\in(0,\fz)$ and $x\in\boz$, let
\begin{align*}
\phi_{t,h}(x):=\begin{cases}
0\ \ \ \ \ \ \ \ \ \ \ \ &\text{if}\ w(x)\le t,\\
w(x)-t\ \ \ &\text{if}\ t<w(x)\le t+h,\\
h\ \ \ \ \ \ \ \ \ \ \ \ &\text{if}\ w(x)>t+h.
\end{cases}
\end{align*}
Then $\phi_{t,h}\in W^{1,2}(\boz)$ and, by \eqref{1.2}, we find that
$$
\int_\boz\nabla(u-v)(x)\cdot\nabla \phi_{t,h}(x)\,dx+\int_\boz V(x)(u-v)(x)\phi_{t,h}(x)\,dx=\int_\boz (f-g)(x)\phi_{t,h}(x)\,dx,
$$
which, combined with the definition of $\phi_{t,h}$ and \eqref{3.4}, implies that
\begin{align}\label{3.5}
&\frac{1}{h}\int_{\{t<w\le t+h\}}|\nabla(u-v)(x)|^2\,dx\\ \nonumber
&\hs=\int_{\{w>t+h\}}(f-g)(x)\,dx
+\frac{1}{h}\int_{\{t<w\le t+h\}}(f-g)(x)[w(x)-t]\,dx\\ \nonumber
&\hs\hs-\int_{\{w>t+h\}}V(x)(u-v)(x)\,dx-
\frac{1}{h}\int_{\{t<w\le t+h\}}V(x)(u-v)(x)[w(x)-t]\,dx.
\end{align}
Letting $h\to0^+$ in \eqref{3.5} and using the Lebesgue dominated
convergence theorem, we conclude that, for almost every $t\in(0,\fz)$,
\begin{align}\label{3.6}
-\frac{d}{dt}\int_{\{w>t\}}|\nabla w(x)|^2\,dx=\int_{\{w>t\}}(f-g)(x)\,dx-\int_{\{w>t\}}V(x)(u-v)(x)\,dx.
\end{align}
Moreover, from Lemma \ref{l3.3}, we deduce that $\|V(u-v)\|_{L^1(\boz)}\ls\|f-g\|_{L^1(\boz)}$,
which, together with \eqref{3.6}, further implies that, for almost every $t\in(0,\fz)$,
\begin{align}\label{3.7}
-\frac{d}{dt}\int_{\{w>t\}}|\nabla w(x)|^2\,dx\ls\|f-g\|_{L^1(\boz)}.
\end{align}
Then, by \eqref{3.7} and Lemma \ref{l3.1}(ii), we know that, for almost every $t\in(0,\fz)$,
\begin{align*}
1\ls-\mu'_{w_+}(t)[\mu_{w_+}(t)]^{-\frac2{n'}}\|f-g\|_{L^1(\boz)},
\end{align*}
which further implies that, for any $t\in(0,\fz)$,
\begin{align*}
t\ls\|f-g\|_{L^1(\boz)}\int_0^t[-\mu'_{w_+}(s)][\mu_{w_+}(s)]^{-\frac2{n'}}\,ds\ls
\|f-g\|_{L^1(\boz)}[\mu_{w_+}(t)]^{\frac{2-n}n}.
\end{align*}
From this and the definition of the decreasing rearrangement, it follows that,
for any $t\in(0,|\boz|/2)$,
\begin{align*}
(w_+)^\ast(t)\ls t^{\frac{2-n}n}\|f-g\|_{L^1(\boz)},
\end{align*}
which further implies that
\begin{align}\label{3.8}
\lf\|(w_+)^\ast\r\|_{L^{\frac{n}{n-2},\fz}(\boz)}\ls \|f-g\|_{L^1(\boz)}.
\end{align}
Repeating the proof of \eqref{3.8}, we find that the estimate in
\eqref{3.8} also holds true for $w_-$, which, together with \eqref{3.8},
implies that \eqref{3.3} holds true in the case that $f,\,g\in L^2(\boz)$.

Let $f,\,g\in L^1(\boz)$. Then, by the definition of the approximable solution,
the Fatou lemma and the fact that \eqref{3.3} holds true in the case that $f,\,g\in L^2(\boz)$,
we conclude that \eqref{3.3} also holds true in this case. This finishes the proof of Lemma \ref{l3.2}.
\end{proof}

Now we show Proposition \ref{p3.2} via Lemma \ref{l3.2}.

\begin{proof}[Proof of Proposition \ref{p3.2}]
We prove this proposition following
the method used in \cite{acmm10,am08,cm14b}.
We first assume that $u$ and $v$ are the solutions to the Dirichlet problem \eqref{1.1}.

Let $w:=u-v$. For any given non-negative integrable function $\psi:\,(0,|\boz|]\to[0,\fz)$
and $t\in[0,|\boz|]$, let
$\Psi(t):=\int_0^t\psi(s)\,ds.$
Furthermore, for any given $r\in[0,|\boz|]$, the function $I$ is defined by setting, for any $t\in[0,|\boz|]$,
\begin{align*}
I_r(t):=\begin{cases}
\Psi(t)\ \ \  &\text{if}\ t\in[0,r],\\
\Psi(r)\ \ \  &\text{if}\ t\in(r,|\boz|].
\end{cases}
\end{align*}
Moreover, the function $\phi$ is defined by setting, for any $x\in\boz$,
$$\phi(x):=\int_0^{w_+(x)}I_r(\mu_{w_+}(t))\,dt.
$$
By the definition of the function $I_r$, we know that the composite function
$I_r\circ\mu_{w_+}$ is bounded, which, together with $w\in W^{1,2}_0(\boz)$ and the chain
rule for derivatives in Sobolev spaces, implies that $\phi\in W^{1,2}_0(\boz)$ and
\begin{align*}
\nabla\phi=\chi_{\{w>0\}}I_r(\mu_{w_+}(w_+))\nabla(u-v)
\end{align*}
holds true almost everywhere in $\boz$. Taking $\phi$
as a test function of \eqref{1.2}, we conclude that
\begin{align}\label{3.9}
&\int_{\{w>0\}}I_r(\mu_{w_+}(w_+(x)))|\nabla(u-v)(x)|^2\,dx+\int_\boz V(x)(u-v)(x)\phi(x)\,dx\\ \nonumber
&\hs=\int_\boz(f-g)(x)\phi(x)\,dx.
\end{align}

Furthermore, it follows, from the definition of the decreasing rearrangement, that,
for any $s\in(0,|\boz|)$,
\begin{align}\label{3.10}
\mu_{w_+}((w_+)^\ast(s))\le s
\end{align}
(see, for example, \cite[(7.1.8)]{pkjf13}), which implies that,
if $\mu_{w_+}(t)=r$, then $t\le(w_+)^\ast(r)$. By this and the definitions of $\phi$ and
$I_r$, we conclude that
\begin{align*}
\|\phi\|_{L^\fz(\boz)}&\le\int_0^\fz I_r(\mu_{w_+}(t))\,dt\le\int_{(w_+)^\ast(r)}^\fz\Psi(\mu_{w_+}(t))\,dt
+\int_0^{(w_+)^\ast(r)}\Psi(r)\,dt\\ \nonumber
&=\int_{(w_+)^\ast(r)}^\fz\int_0^{\mu_{w_+}(t)}\psi(s)\,ds\,dt+\Psi(r)(w_+)^\ast(r)\\ \nonumber
&=\int_0^r[(w_+)^\ast(s)-(w_+)^\ast(r)]\psi(s)\,ds+\Psi(r)(w_+)^\ast(r)
=\int_0^r(w_+)^\ast(s)\psi(s)\,ds,
\end{align*}
which, combined with \eqref{3.2} and the definition of $L^{n/(n-2),\fz}(\boz)$, further implies that
\begin{align}\label{3.11}
\|\phi\|_{L^\fz(\boz)}\ls\|f-g\|_{L^1(\boz)}\int_0^r\psi(s)s^{-\frac{n-2}n}\,ds.
\end{align}

Recall that, for any measurable function $u$ on $\boz$,
the \emph{increasing rearrangement} $u_\ast:\ [0,|\boz|]\to[0,\fz]$ of $u$ is defined by setting, for any
$s\in[0,|\boz|]$, $u_\ast(s):=u^\ast(|\boz|-s)$.
Then the Hardy-Littlewood inequality states that, for any measurable functions $u$ and $v$ on $\boz$,
\begin{align}\label{3.12}
\int_0^{|\boz|}u^\ast(s)v_\ast(s)\,ds\le\int_{\boz}|u(x)v(x)|\,dx\le\int_0^{|\boz|} u^\ast(s)v^{\ast}(s)\,ds
\end{align}
(see, for example, \cite[Theorem 2.2]{bs88}).
Moreover, from \eqref{3.10} and the facts that $w_+$ and $(w_+)^\ast$
are equivalently distributed functions and $I_r$ is increasing,
we deduce that, for any $t\in(0,|\boz|)$,
$$(I_r\circ\mu_{w_+}\circ w_+)_\ast(t)=(I_r\circ\mu_{w_+}\circ(w_+)^\ast)_\ast(t)\ge (I_r)_\ast(t)=I_r(t)
$$
(see also \cite[(4.30)]{acmm10}), which, together
with \eqref{3.12}, implies that
\begin{align*}
&\int_{\{w>0\}}I_r(\mu_{w_+}(w_+(x)))|\nabla(u-v)(x)|^2\,dx\\ \nonumber
&\hs\gs\int_0^{|\boz|}[|\nabla w_+|^\ast(t)]^2
(I_r\circ\mu_{w_+}\circ(w_+)^\ast)_\ast(t)\,dt\\ \nonumber
&\hs\gs\int_0^{|\boz|}[|\nabla w_+|^\ast(t)]^2I_r(t)\,dt\gs\int_0^r
[|\nabla w_+|^\ast(t)]^2I_r(t)\,dt\\ \nonumber
&\hs\gs[|\nabla w_+|^\ast(r)]^2\int_0^r\int_0^t\psi(s)\,ds\,dt\sim
[|\nabla w_+|^\ast(r)]^2\int_0^r\psi(s)(r-s)\,ds.
\end{align*}
By this, \eqref{3.9}, \eqref{3.11} and Lemma \ref{l3.3}, we conclude that
\begin{align*}
&[|\nabla w_+|^\ast(r)]^2\int_0^r\psi(s)(r-s)\,ds\\ \nonumber
&\hs\ls\|f-g\|^2_{L^1(\boz)}\int_0^r\psi(s)s^{-\frac{n-2}n}\,ds\\ \nonumber
&\hs\hs+\|V(u-v)\|_{L^1(\boz)}\|f-g\|_{L^1(\boz)}\int_0^r\psi(s)s^{-\frac{n-2}n}\,ds\\ \nonumber
&\hs\ls\|f-g\|^2_{L^1(\boz)}\int_0^r\psi(s)s^{-\frac{n-2}n}\,ds,
\end{align*}
which implies that, for any $r\in(0,|\boz|]$,
\begin{align}\label{3.13}
[|\nabla w_+|^\ast(r)]^2\sup_{\psi}\dfrac{\int_0^r\psi(s)(r-s)\,ds}
{\int_0^r\psi(s)s^{-(n-2)/n}\,ds}\ls\|f-g\|^2_{L^1(\boz)},
\end{align}
where the supremum is taken over all non-negative integrable functions $\psi$ on the interval $(0,|\boz|]$.
Moreover, it is easy to see that, for any $r\in(0,|\boz|]$,
$$\int_0^r(r-s)\,ds\sim r^{\frac{2}{n'}}\int_0^rs^{-\frac{n-2}n}\,ds,
$$
which further implies that
$$\sup_{\psi}\dfrac{\int_0^r\psi(s)(r-s)\,ds}{\int_0^r\psi(s)s^{-(n-2)/n}\,ds}\gs r^{\frac{2}{n'}}.
$$
From this and \eqref{3.13},  it follows that, for any $r\in(0,|\boz|]$,
\begin{align*}
|\nabla w_+|^\ast(r)r^{\frac{1}{n'}}\ls\|f-g\|_{L^1(\boz)}
\end{align*}
and hence
\begin{align}\label{3.14}
\|\nabla w_+\|_{L^{n',\fz}(\boz)}\ls\|f-g\|_{L^1(\boz)}.
\end{align}
Repeating the proof of \eqref{3.14} with replacing $w_+$ by $w_-$, we find that
the estimate in \eqref{3.14} also holds true for $w_-$, which, together with \eqref{3.14},
implies that \eqref{3.1} holds true for $w$.
This finishes the proof for the case of the Dirichlet problem \eqref{1.1}.

Now we assume that $u$ and $v$ are the solutions to the Neumann problem \eqref{1.3}.
In this case, let $w:=(u-v)-\mathrm{med}(u-v)$.
Then, by repeating the proof of \eqref{3.14} with replacing \eqref{3.2} with \eqref{3.3},
we know that $\|\nabla w_\pm\|_{L^{n',\fz}(\boz)}\ls\|f-g\|_{L^1(\boz)}$,
which implies that \eqref{3.1} holds true in this case.
This finishes the proof of Proposition \ref{p3.2}.
\end{proof}

To give the proof of Theorem \ref{t1.1}, we need the following notion of the $K$-functional. We
begin with the quasi-normed function space.
Let $\boz\subset\rn$ be a bounded domain.
A \emph{quasi-normed function space} $X(\boz)$ on $\boz$ is defined to be
a linear space of all measurable functions on
$\boz$ equipped with a quasi-norm having the following properties:
\begin{itemize}
  \item[(i)] $\|u\|_{X(\boz)}>0$ if $u\not\equiv0$;
  for any $\lz\in\rr$ and $u\in X(\boz)$, $\|\lz u\|_{X(\boz)}=|\lz|\|u\|_{X(\boz)}$;
  there exists a positive constant $C\in[1,\fz)$ such that, for any $u,\,v\in X(\boz)$,
  $$\|u+v\|_{X(\boz)}\le C[\|u\|_{X(\boz)}+\|v\|_{X(\boz)}];$$
  \item[(ii)] $0\le v\le u$ almost everywhere implies that $\|v\|_{X(\boz)}\le\|u\|_{X(\boz)}$;
  \item[(iii)] $0\le u_k\uparrow u$ almost everywhere as $k\to\fz$ implies that $\|u_k\|_{X(\boz)}\uparrow\|u\|_{X(\boz)}$
  as $k\to\fz$;
  \item [(iv)] if $E\subset\boz$ is measurable, then $\|\chi_E\|_{X(\boz)}<\fz$
  and there exists a positive constant
  $C$ such that, for any $u\in X(\boz)$, $\int_E |u(x)|\,dx\le C\|u\|_{X(\boz)}$.
\end{itemize}
Moreover, the space $X(\boz)$ is called a \emph{Banach function space} if (i) holds true with $C=1$.

Let $X_1(\boz)$ and $X_2(\boz)$ be two quasi-normed function spaces on $\boz$. Then, for any $s\in(0,\fz)$
and $u\in X_1(\boz)+X_2(\boz)$, the \emph{$K$-functional} $K(s,u;X_1(\boz),X_2(\boz))$ is defined as
$$K(s,u;X_1(\boz),X_2(\boz)):=\inf\{\|u_1\|_{X_1(\boz)}+s\|u_2\|_{X_2(\boz)}:\ u=u_1+u_2\}.
$$
Let $m\in\nn$. Then, for any $s\in(0,\fz)$ and vector-valued measurable function $U:\ \boz\to \rr^m$ such that
$U\in[X_1(\boz)]^m+[X_2(\boz)]^m$, the \emph{$K$-functional}
$K(s,U;[X_1(\boz)]^m,[X_2(\boz)]^m)$ is defined by setting
$$K(s,U;[X_1(\boz)]^m,[X_2(\boz)]^m):=\inf\{\|U_1\|_{X_1(\boz)}+s\|U_2\|_{X_2(\boz)}:\ U=U_1+U_2\}.
$$
It is easy to see that, for any $s\in(0,\fz)$ and $U\in[X_1(\boz)]^m+[X_2(\boz)]^m$,
\begin{align}\label{3.15}
K(s,|U|;X_1(\boz),X_2(\boz))\sim K(s,U;[X_1(\boz)]^m,[X_2(\boz)]^m).
\end{align}
We point out that the $K$-functional was introduced by Peetre \cite{p63}
in the constructions of families of intermediate function spaces
(see, for example, \cite[Chapter 5]{bs88} for the details about the $K$-functional).

We now prove Theorem \ref{t1.1} by using Propositions \ref{p3.1}
and \ref{p3.2} and the method of the $K$-functional.

\begin{proof}[Proof of Theorem \ref{t1.1}]
We only give out the proof for the case of the Neumann problem,
since the proof for the Dirichlet problem is similar, the details being omitted.

Let
$$T:\ L^1(\boz)\to\lf[L^{n',\fz}(\boz)\r]^n,\ \ Tf:=\nabla u,
$$
where, for any $f\in L^1(\boz)$, $u$ is the weak solution of the Neumann problem
\eqref{1.3}. Fix $f\in L^1(\boz)$. Then, for any decomposition $f=f_0+f_1$ with $f_0\in L^{n,1}(\boz)$
and $f_1\in L^1(\boz)$, we have the decomposition
\begin{align*}
Tf=Tf_0+(Tf-Tf_0).
\end{align*}
By the definition of the $K$-functional and Propositions \ref{p3.1} and \ref{p3.2}, we know that,
for any $s\in(0,\fz)$,
\begin{align*}
&K\lf(s,Tf;\lf[L^{n',\fz}(\boz)\r]^n,\lf[L^\fz(\boz)\r]^n\r) \\ \nonumber
&\hs\le\|Tf-Tf_0\|_{L^{n',\fz}(\boz)}+s\|Tf_0\|_{L^\fz(\boz)}\\ \nonumber
&\hs\ls\|f-f_0\|_{L^1(\boz)}+s\|f_0\|_{L^{n,1}(\boz)}\sim\|f_1\|_{L^1(\boz)}+s\|f_0\|_{L^{n,1}(\boz)},
\end{align*}
which further implies that, for any $s\in(0,\fz)$,
\begin{align}\label{3.16}
K\lf(s,Tf;\lf[L^{n',\fz}(\boz)\r]^n,\lf[L^\fz(\boz)\r]^n\r)\ls
K\lf(s,f;L^1(\boz),L^{n,1}(\boz)\r).
\end{align}
Moreover, from \eqref{3.15} and \cite[(4.8)]{h70}, we deduce that, for any $s\in(0,\fz)$,
\begin{align*}
K\lf(s,Tf;\lf[L^{n',\fz}(\boz)\r]^n,\lf[L^\fz(\boz)\r]^n\r)&\sim
K\lf(s,|Tf|;L^{n',\fz}(\boz),L^\fz(\boz)\r)\\ \nonumber
&\sim\sup_{t\in(0,s^{n/(n-1)}]}\lf\{t^{\frac{1}{n'}}(Tf)^\ast(t)\r\}
\end{align*}
and hence
\begin{align}\label{3.17}
K\lf(s^{\frac{1}{n'}},Tf;\lf[L^{n',\fz}(\boz)\r]^n,\lf[L^\fz(\boz)\r]^n\r)
\sim\sup_{t\in(0,s]}\lf\{t^{\frac{1}{n'}}(Tf)^\ast(t)\r\}.
\end{align}

Furthermore, by \cite[Theorem 4.2]{h70}, we conclude that, for any $s\in(0,\fz)$ and $f\in L^1(\boz)$,
\begin{align*}
K\lf(s,f;L^1(\boz),L^{n,1}(\boz)\r)\sim\int_0^{s^{n'}}f^\ast(t)\,dt
+s\int_{s^{n'}}^{\fz}f^\ast(t)t^{-\frac1{n'}}\,dt,
\end{align*}
which, together with \eqref{3.16} and \eqref{3.17}, implies that, for any $s\in(0,\fz)$,
\begin{align*}
|\nabla u|^\ast(s)&=(Tf)^\ast(s)\ls
s^{-\frac{1}{n'}}K\lf(s^{\frac{1}{n'}},Tf;\lf[L^{n',\fz}(\boz)\r]^n,\lf[L^\fz(\boz)\r]^n\r)\\
&\ls s^{-\frac{1}{n'}}\lf\{\int_0^{s}f^\ast(t)\,dt+s^{\frac1{n'}}
\int_{s}^{\fz}f^\ast(t)t^{-\frac1{n'}}\,dt\r\}\\
&\sim s^{-\frac1{n'}}\int_0^sf^\ast(t)\,dt+\int_s^{\fz}f^\ast(t)t^{-\frac1{n'}}\,dt.
\end{align*}
From this and the fact that, for any $s\in[|\boz|,\fz)$, $f^\ast(s)=0=|\nabla u|^\ast(s)$,
it follows that, for any $s\in(0,|\boz|)$,
$$|\nabla u|^\ast(s)\ls s^{-\frac1{n'}}\int_0^sf^\ast(t)\,dt+\int_s^{|\boz|}f^\ast(t)t^{-\frac1{n'}}\,dt,
$$
which completes the proof of Theorem \ref{t1.1}.
\end{proof}

\section{Proof of Theorem \ref{t1.2}\label{s4}}
\hskip\parindent In this section, we give out the proof of Theorem \ref{t1.2}.
We first recall some necessary notation.

A quasi-normed function space $X(\boz)$ is said to be \emph{rearrangement invariant} if there exists
a quasi-normed function space $\overline{X}(0,|\boz|)$ on the interval $(0,|\boz|)$,
called the \emph{representation space} of $X(\boz)$, having the property that, for any $u\in X(\boz)$,
$\|u\|_{X(\boz)}=\|u^\ast\|_{\overline{X}(0,|\boz|)}.$
Obviously, if $X(\boz)$ is a rearrangement invariant and $u^\ast=v^\ast$, then
$\|u\|_{X(\boz)}=\|v\|_{X(\boz)}$.

We point out that the classical (weighted) Lebesgue space, (weighted) Lorentz space,
(weighted) Lorentz-Zygmund space and Orlicz space are all rearrangement invariant
quasi-normed function spaces (see, for example, \cite{bs88}).

As a corollary of Theorem \ref{t1.1}, we have the following general criterion for
the rearrangement invariant quasi-normed function space.

\begin{lemma}\label{l4.1}
Let $n\ge3$, $\boz$ be a bounded domain in $\rn$,
$0\le V\in RH_n(\rn)$ and $V\not\equiv0$ on $\boz$.
Assume that $\partial\boz\in W^2L^{n-1,\,1}$ or $\boz$ is semi-convex.
Let $X(\boz)$ and $Y(\boz)$ be rearrangement invariant quasi-normed spaces on $\boz$,
$\overline{X}(0,|\boz|)$ and $\overline{Y}(0,|\boz|)$ be the representation spaces of
$X(\boz)$, respectively, $Y(\boz)$.
Assume that $f\in X(\boz)$ and there exists a positive constant $C$ such that, for any $s\in(0,|\boz|)$,
\begin{align}\label{4.1}
\lf\|s^{-\frac1{n'}}\int_0^sf^\ast(t)\,dt\r\|_{\overline{Y}(0,|\boz|)}\le
C\|f^\ast\|_{\overline{X}(0,|\boz|)}
\end{align}
and
\begin{align}\label{4.2}
\lf\|\int_s^{|\boz|}f^\ast(t)t^{-\frac1{n'}}\,dt\r\|_{\overline{Y}(0,|\boz|)}\le C\|f^\ast\|_{\overline{X}(0,|\boz|)}.
\end{align}
If $u$ is a weak solution to the Dirichlet
problem \eqref{1.1} or the Neumann problem \eqref{1.3},
then there exists a positive constant $C$, depending on $n$, $[V]_{RH_n(\rn)}$ and $\boz$, such that
\begin{align*}
\|\nabla u\|_{Y(\boz)}\le C\|f\|_{X(\boz)}.
\end{align*}
\end{lemma}

\begin{proof}
By the definition of $X(\boz)$, we know that $X(\boz)\subset L^1(\boz)$,
which, together with $f\in X(\boz)$, implies that $f\in L^1(\boz)$.
From this, Theorem \ref{t1.1}, \eqref{4.1} and \eqref{4.2}, it follows that
\begin{align*}
\|\nabla u\|_{Y(\boz)}&=\||\nabla u|^\ast\|_{\overline{Y}(0,|\boz|)}\ls
\lf\|s^{-\frac1{n'}}\int_0^sf^\ast(t)\,dt\r\|_{\overline{Y}(0,|\boz|)}+
\lf\|\int_s^{|\boz|}f^\ast(t)t^{-\frac1{n'}}\,dt\r\|_{\overline{Y}(0,|\boz|)}\\
&\ls\|f^\ast\|_{\overline{X}(0,|\boz|)}\sim\|f\|_{X(\boz)}.
\end{align*}
This finishes the proof of Lemma \ref{l4.1}.
\end{proof}

Moreover, to prove Theorem \ref{t1.2}, we need the following weighted Hardy type inequality in
Lebesgue spaces. More precisely, Lemma \ref{l4.2} was
established in \cite[Section 1.3.2, Theorems 2 and 3]{m11}
and Lemma \ref{l4.3} was obtained in \cite[Proposition 2.9]{cs93}.

\begin{lemma}\label{l4.2}
Let $\boz\subset\rn$ be a bounded domain,
$w,\,v$ two non-negative Borel functions on $(0,|\boz|]$ and $1\le p\le q\le\fz$.

{\rm(i)} The inequality
$$\lf\{\int_0^{|\boz|}\lf|w(s)\int_0^sf(t)\,dt\r|^q\,ds\r\}^{\frac1q}\le C\lf\{\int_0^{|\boz|}|v(s)f(s)|^p\,ds\r\}^{\frac1p}
$$
holds true for any Borel function $f$ on $(0,|\boz|]$, where $C$ is a positive
constant independent of $f$, if and only if
\begin{align}\label{4.3}
\sup_{r\in(0,|\boz|]}\lf[\int_r^{|\boz|}[w(x)]^q\,dx\r]^{\frac1q}
\lf\{\int_0^r[v(x)]^{-p'}\,dx\r\}^{\frac1{p'}}<\fz.
\end{align}

{\rm(ii)} The inequality
\begin{align*}
\lf\{\int_0^{|\boz|}\lf|w(s)\int_s^{|\boz|} f(t)\,dt\r|^q\,ds\r\}^{\frac1q}\le
C\lf\{\int_0^{|\boz|}|v(s)f(s)|^p\,ds\r\}^{\frac1p}
\end{align*}
holds true for any Borel function $f$ on $(0,|\boz|]$, where $C$ is a positive
constant independent of $f$, if and only if
\begin{align}\label{4.4}
\sup_{r\in(0,|\boz|]}\lf\{\int_0^r[w(x)]^q\,dx\r\}^{\frac1q}
\lf\{\int_r^{|\boz|}[v(x)]^{-p'}\,dx\r\}^{\frac1{p'}}<\fz.
\end{align}

\end{lemma}

\begin{lemma}\label{l4.3}
Let $\boz\subset\rn$ be a bounded domain and
$w,\,v$ two non-negative Borel functions on $(0,|\boz|]$.
Assume that $p_0\in(0,1]$ and $p_1\in[p_0,\fz)$.

{\rm(i)} The inequality
\begin{align*}
\lf\{\int_0^{|\boz|}\lf[s^{-\frac1{n'}}\int_0^s
f^\ast(t)\,dt\r]^{p_1}w(s)\,ds\r\}^{\frac1{p_1}}\le C\lf\{\int_0^{|\boz|} [f^\ast(s)]^{p_0}v(s)\,ds\r\}^{\frac1{p_0}}
\end{align*}
holds true for any Borel function $f$ on $\boz$, where $C$ is a positive
constant independent of $f$, if and only if
\begin{align}\label{4.5}
\sup_{t\in(0,|\boz|]}\lf\{\int_0^ts^{\frac{p_1}n}w(s)\,ds
+t^{p_1}\int_t^{|\boz|}w(s)s^{-\frac{p_1}{n'}}\,ds\r\}^{\frac1{p_1}}
\lf\{\int_0^tv(s)\,ds\r\}^{-\frac1{p_0}}<\fz.
\end{align}

{\rm(ii)} The inequality
\begin{align*}
\lf\{\int_0^{|\boz|}\lf[\int_s^{|\boz|} f^\ast(t)t^{-\frac1{n'}}\,dt\r]^{p_1}w(s)\,ds\r\}^{\frac1{p_1}}\le C\lf\{\int_0^{|\boz|} [f^\ast(s)]^{p_0}v(s)\,ds\r\}^{\frac1{p_0}}
\end{align*}
holds true for any Borel function $f$ on $\boz$, where $C$ is a positive
constant independent of $f$, if and only if
\begin{align}\label{4.6}
\sup_{t\in(0,|\boz|]}\lf\{\int_0^t(t-s)^{\frac{p_1}n}w(s)\,ds\r\}^{\frac1{p_1}}
\lf\{\int_0^tv(s)\,ds\r\}^{-\frac1{p_0}}<\fz.
\end{align}

\end{lemma}

Now we give out the proof of Theorem \ref{t1.2} by using Theorem \ref{t1.1},
Lemmas \ref{l4.1}, \ref{l4.2} and \ref{l4.3}.

\begin{proof}[Proof of Theorem \ref{t1.2}]
(i) By Theorem \ref{t1.1} and the H\"older inequality, we find that, for any $p\in[1,n/(n-1))$,
\begin{align*}
\|\nabla u\|_{L^p(\boz)}=\||\nabla u|^\ast\|_{L^p(0,|\boz|)}\ls&
\lf\{\int_0^{|\boz|}s^{-\frac{p}{n'}}\lf[\int_0^{s}f^\ast(t)\,dt\r]^p\,ds\r\}^{\frac1p}\\
&\hs+\lf\{\int_0^{|\boz|}\lf[\int_{s}^{|\boz|}f^\ast(t)t^{-\frac1{n'}}\,dt\r]^p\,ds\r\}^{\frac1p}\\
\ls& \|f\|_{L^1(\boz)}\lf[\int_0^{|\boz|}s^{-\frac{p}{n'}}\,ds\r]^{\frac1p}\sim\|f\|_{L^1(\boz)}.
\end{align*}
Thus, (i) holds true.

(ii) Let $k\in(0,1]$. Then, from Theorem \ref{t1.1} and Remark \ref{r1.5}(ii), it follows that
\begin{align*}
\|\nabla u\|_{L^{\frac{n}{n-1},\fz}(\boz)}=& \sup_{s\in(0,|\boz|)}\lf\{s^{\frac{1}{n'}}|\nabla u|^\ast(s)\r\}\\
\ls& \sup_{s\in(0,|\boz|)}\lf\{\int_0^sf^\ast(t)\,dt+s^{-\frac1{n'}}
\int_s^{|\boz|}f^\ast(t)t^{-\frac1{n'}}\,dt\r\}\\
\ls& \|f\|_{L^1(\boz)}\ls\|f\|_{L^{1,k}(\boz)},
\end{align*}
which completes the proof of (ii).

(iii) Let $q\in(1,n)$. Then
$$\sup_{t\in(0,|\boz|]}\lf\{\int_t^{|\boz|}\lf[s^{-\frac1{n'}}\r]^{\frac{nq}{n-q}}\,ds\r\}^{\frac{n-q}{nq}}
\lf\{\int_0^t\,ds\r\}^{\frac1{q'}}
\ls\sup_{t\in(0,|\boz|]}\lf\{t^{-\frac{q-1}{q}}t^{\frac{1}{q'}}\r\}\ls1,
$$
which, together with Lemma \ref{l4.2}(i), implies that
\begin{align}\label{4.7}
\lf\{\int_0^{|\boz|}\lf[s^{-\frac1{n'}}\int_0^s f^\ast(t)\,dt\r]^{\frac{nq}{n-q}}\,ds\r\}^{\frac{n-q}{nq}}&\ls
\lf\{\int_0^{|\boz|}[f^\ast(t)]^{q}\,dt\r\}^{\frac1q}.
\end{align}
Moreover,
\begin{align*}
\sup_{t\in(0,|\boz|]}\lf\{\int_0^t\,ds\r\}^{\frac{n-q}{nq}}\lf\{\int_t^{|\boz|}\lf[s^{\frac{1}{n'}}\r]^{-q'}
\,ds\r\}^{\frac1{q'}}\ls
\sup_{t\in(0,|\boz|]}\lf\{t^{\frac{n-q}{nq}}t^{\frac{q-n}{n(q-1)}\frac{q-1}{q}}\r\}\ls1,
\end{align*}
which, combined with Lemma \ref{l4.2}(ii), further implies that
\begin{align*}
\lf\{\int_0^{|\boz|}\lf[\int_s^{|\boz|} f^\ast(t)t^{-\frac1{n'}}\,dt\r]^{\frac{nq}{n-q}}\,ds\r\}^{\frac{n-q}{nq}}\ls
\lf\{\int_0^{|\boz|}\lf[f^\ast(t)t^{-\frac1{n'}}\r]^{q}t^{\frac q{n'}}\,dt\r\}^{\frac1q}\sim\|f^\ast\|_{L^q(0,|\boz|)}.
\end{align*}
By this, \eqref{4.7} and Lemma \ref{l4.1}, we conclude that
\begin{align}\label{4.8}
\|\nabla u\|_{L^{\frac{nq}{n-q}}(\boz)}=\||\nabla u|^\ast\|_{L^{\frac{nq}{n-q}}(0,|\boz|)}\ls
\|f^\ast\|_{L^q(0,|\boz|)}\sim\|f\|_{L^q(\boz)}.
\end{align}
This finishes the proof of (iii).

(iv) Let $q\in(1,n)$ and $k\in(0,\fz]$.
When $k\in(0,1]$, it is easy to see that
\begin{align*}
\sup_{t\in(0,|\boz|]}\lf\{\int_0^t s^{\frac{k}{n}}s^{(\frac{1}{q}-\frac 1n)k-1}\,ds
+t^k\int_t^{|\boz|}s^{\frac kq-k-1}\,ds\r\}^{\frac1k}
\lf\{\int_0^ts^{\frac{k}{q}-1}\,ds\r\}^{-\frac1k}\ls1
\end{align*}
and
\begin{align*}
&\sup_{t\in(0,|\boz|]}\lf\{\int_0^t(t-s)^{\frac kn}s^{(\frac1q-\frac1n)k-1}\,ds\r\}^{\frac1k}
\lf\{\int_0^ts^{\frac{k}{q}-1}\,ds\r\}^{-\frac1k}\\
&\hs\ls\sup_{t\in(0,|\boz|]}\lf\{t^{\frac kn}\int_0^ts^{(\frac1q-\frac1n)k-1}\,ds\r\}^{\frac1k}
\lf\{\int_0^ts^{\frac{k}{q}-1}\,ds\r\}^{-\frac1k}\ls1,
\end{align*}
which imply that \eqref{4.5} and \eqref{4.6} hold true for $p_0=k=p_1$,
$w(s):=s^{(1/q-1/n)k-1}$ and $v(s):=s^{k/q-1}$ with $s\in(0,|\boz|]$. From this and
Lemma \ref{l4.3}, it follows that
\begin{align*}
\lf\{\int_0^{|\boz|}\lf[s^{-\frac1{n'}}\int_0^sf^\ast(t)\,dt\r]^k
s^{(\frac1q-\frac1n)k-1}\,ds\r\}^{\frac1k}&\ls
\lf\{\int_0^{|\boz|}\lf[f^\ast(s)\r]^ks^{\frac kq-1}\,ds\r\}^{\frac1k}\\
&\sim\lf\|(\cdot)^{\frac1q-\frac1k}f^\ast(\cdot)\r\|_{L^k(0,|\boz|)}
\end{align*}
and
\begin{align*}
\lf\{\int_0^{|\boz|}\lf[\int_{s}^{|\boz|}f^\ast(t)t^{-\frac1{n'}}\,dt\r]^k
s^{(\frac1q-\frac1n)k-1}\,ds\r\}^{\frac1k}
\ls\lf\|(\cdot)^{\frac1q-\frac1k}f^\ast(\cdot)\r\|_{L^k(0,|\boz|)},
\end{align*}
which, together with Lemma \ref{l4.1}, further implies that
\begin{align}\label{4.9}
\|\nabla u\|_{L^{\frac{nq}{n-q},k}(\boz)}&=\lf\|(\cdot)^{\frac{n-q}{nq}-\frac1k}
|\nabla u|^\ast(\cdot)\r\|_{L^k(0,|\boz|)}\\ \nonumber
&\ls\lf\|(\cdot)^{\frac1q-\frac1k}f^\ast(\cdot)\r\|_{L^k(0,|\boz|)}\sim\|f\|_{L^{q,k}(\boz)}.
\end{align}
Thus, in the case that $k\in(0,1]$, the conclusion of (iv) holds true.

Moreover, when $k\in(1,\fz]$, replacing Lemma \ref{l4.3} by Lemma \ref{l4.2} and
repeating the proof of \eqref{4.9}, we conclude that the conclusion of (iv)
holds true in the case that $k\in(1,\fz]$. This finishes the proof of (iv).

(v) It is easy to see that, for any $p\in[1,\fz)$,
$$\sup_{t\in(0,|\boz|]}\lf\{\int_t^{|\boz|}s^{-\frac{p}{n'}}\,ds\r\}^{\frac1p}
\lf\{\int_0^t\,ds\r\}^{\frac1{n'}}
\ls1
$$
and
\begin{align*}
\sup_{t\in(0,|\boz|]}\lf\{\int_0^t\,ds\r\}^{\frac1p}\lf\{\int_t^{|\boz|}
\lf[s^{\frac{1}{n'}}\r]^{-n'}\,ds\r\}^{\frac1{n'}}\ls
\sup_{t\in(0,|\boz|]}\lf\{t^{\frac1p}\lf[\log\lf(\frac{|\boz|}{t}\r)\r]^{\frac1{n'}}\r\}\ls1.
\end{align*}
By this and similar to the proof of (iii), we know that (v) holds true.

(vi) We first assume that $k\in(1,\fz)$. Then we know that
\begin{align}\label{4.10}
&\sup_{t\in(0,|\boz|]}\lf[\int_t^{|\boz|}\lf\{s^{-(\frac1k+\frac{1}{n'})}\lf[1+\log\lf(\frac{|\boz|}{s}\r)\r]^{-1}
\r\}^k\,ds\r]^{\frac1k}
\lf\{\int_0^t\lf[s^{\frac{1}{n}-\frac1k}\r]^{-k'}\,ds\r\}^{\frac{1}{k'}}\\ \nonumber
&\hs\ls\sup_{t\in(0,|\boz|]}\lf\{\lf[\int_t^{|\boz|}s^{-1-\frac{k}{n'}}\,ds\r]^{\frac1k}
t^{\frac1{n'}}\r\}\ls1.
\end{align}
Furthermore, by the change of variables, we find that, for any $t\in(0,|\boz|)$,
\begin{align*}
\int_0^ts^{-1}\lf[1+\log\lf(\frac{|\boz|}{s}\r)\r]^{-k}\,ds
&=\int_{|\boz|/t}^\fz s^{-1}(1+\log s)^{-k}\,ds\le\sum_{j=1}^\fz\lf[j+\log\lf(\frac{|\boz|}t\r)\r]^{-k}\\
&\le\sum_{j=1}^\fz\int_{j-1}^j\lf[s+\log\lf(\frac{|\boz|}t\r)\r]^{-k}\,ds\\
&=\int_{0}^\fz\lf[s+\log\lf(\frac{|\boz|}t\r)\r]^{-k}\,ds=\frac{1}{k-1}
\lf[\log\lf(\frac{|\boz|}{t}\r)\r]^{1-k},
\end{align*}
which further implies that
\begin{align*}
\sup_{t\in(0,|\boz|]}\lf[\int_0^t\lf\{s^{-\frac1k}
\lf[1+\log\lf(\frac{|\boz|}{s}\r)\r]^{-1}\r\}^k\,ds\r]^{\frac1k}
\lf\{\int_t^{|\boz|}\lf[s^{\frac{1}{k'}}\r]^{-k'}\,ds\r\}^{\frac{1}{k'}}
\ls1.
\end{align*}
From this and \eqref{4.10}, we deduce that \eqref{4.3}
(taking $p=k=q$,
$$w(s):=s^{-(\frac1k+\frac{1}{n'})}\lf[1+\log\lf(\frac{|\boz|}{s}\r)\r]^{-1}$$
and $v(s):=s^{1/n-1/k}$ with $s\in(0,|\boz|]$ in \eqref{4.3}) and \eqref{4.4}
(taking $p=k=q$,
$$w(s):=s^{-\frac1k}\lf[1+\log\lf(\frac{|\boz|}s\r)\r]^{-1}$$
and $v(s):=s^{1/k'}$ with $s\in(0,|\boz|]$ in \eqref{4.4})
hold true, which, together with Lemma \ref{l4.2}, implies that
\begin{align*}
&\lf[\int_0^{|\boz|}\lf\{s^{-(\frac1k+\frac{1}{n'})}\lf[1+\log\lf(\frac{|\boz|}{s}\r)\r]^{-1}
\int_0^sf^\ast(t)\,dt\r\}^k
\,ds\r]^{\frac1k}\\
&\hs\ls\lf\{\int_0^{|\boz|}\lf[f^\ast(s)\r]^ks^{\frac kn-1}\,ds\r\}^{\frac1k}\sim
\lf\|(\cdot)^{\frac1n-\frac1k}f^\ast(\cdot)\r\|_{L^k(0,|\boz|)}
\end{align*}
and
\begin{align*}
&\lf[\int_0^{|\boz|}\lf\{s^{-\frac1k}\lf[1+\log\lf(\frac{|\boz|}{s}\r)\r]^{-1}
\int_{s}^{|\boz|}f^\ast(t)t^{-\frac1{n'}}\,dt\r\}^k\,ds\r]^{\frac1k}\\ \nonumber
&\hs\ls\lf\{\int_0^{|\boz|}\lf[f^\ast(s)\r]^ks^{\frac kn-1}\,ds\r\}^{\frac1k}
\sim\lf\|(\cdot)^{\frac1n-\frac1k}f^\ast(\cdot)\r\|_{L^k(0,|\boz|)}.
\end{align*}
By this and Lemma \ref{l4.1}, we conclude that
\begin{align*}
\|\nabla u\|_{L^{\fz,k}(\log L)^{-1}(\boz)}
\ls\lf\|(\cdot)^{\frac1n-\frac1k}f^\ast(\cdot)\r\|_{L^k(0,|\boz|)}\sim\|f\|_{L^{n,k}(\boz)}.
\end{align*}
Thus, (vi) holds true in the case that $k\in(1,\fz)$.

Moreover, when $k=\fz$, it is easy to see that
\begin{align*}
&\sup_{s\in(0,|\boz|)}\lf[1+\log\lf(\frac{|\boz|}{s}\r)\r]^{-1}
\lf[s^{-\frac1{n'}}\int_{0}^s f^\ast(r)\,dr+\int_s^{|\boz|}f^\ast(r)r^{-\frac{1}{n'}}\,dr\r]\\
&\hs\ls\|f\|_{L^{n,\fz}(\boz)}\sup_{s\in(0,|\boz|)}\lf[1+\log\lf(\frac{|\boz|}{s}\r)\r]^{-1}
\lf[s^{-\frac{1}{n'}}\int_0^s r^{-\frac1n}\,dr+\int_s^{|\boz|}r^{-1}\,dr\r]\\
&\hs\ls\|f\|_{L^{n,\fz}(\boz)},
\end{align*}
which implies that $\|\nabla u\|_{L^{\fz,\fz}(\log L)^{-1}(\boz)}\ls\|f\|_{L^{n,\fz}(\boz)}$.
This finishes the proof of the case that $k=\fz$ and hence (vi).

(vii) Let $q\in(n,\fz]$. Then, from Proposition \ref{p3.1} and Remark \ref{r1.5}(iii), it follows that
$$\|\nabla u\|_{L^\fz(\boz)}\ls\|f\|_{L^{n,1}(\boz)}\ls\|f\|_{L^q(\boz)},
$$
which completes the proof of (vii).

(viii) By (ii) and (iii) of Remark \ref{r1.5}, we find that $L^{q,k}(\boz)\subset L^{n,1}(\boz)$
when $q=n$ and $k\in(0,1]$ or $q\in(n,\fz]$ and $k\in(0,\fz]$, which, together with Proposition
\ref{p3.1}, implies that (viii) holds true.
This finishes the proof of Theorem \ref{t1.2}.
\end{proof}

\smallskip

{\bf Acknowledgement.} The authors would like to thank Professor Renjin Jiang
for some helpful discussions on this topic and, indeed, the examples given in Remarks \ref{r1.2}(i)
and \ref{r1.3}(i) are attributed to him. The first author would also like to
thank Professor Jun Geng for some helpful discussions on this topic.

\bigskip

\noindent Sibei Yang

\medskip

\noindent School of Mathematics and Statistics, Gansu Key Laboratory of Applied Mathematics and
Complex Systems, Lanzhou University, Lanzhou, Gansu 730000, People's Republic of China

\smallskip

\noindent{\it E-mail:} \texttt{yangsb@lzu.edu.cn}

\bigskip

\noindent Der-Chen Chang

\medskip

\noindent Department of Mathematics and Department of Computer Science,
Georgetown University, Washington D.C. 20057, USA

\noindent Department of Mathematics, Fu Jen Catholic University, Taipei 242, Taiwan
\smallskip

\noindent{\it E-mail:} \texttt{chang@georgetown.edu}

\bigskip

\noindent Dachun Yang (Corresponding author)

\medskip

\noindent School of Mathematical Sciences, Beijing Normal
University, Laboratory of Mathematics and Complex Systems, Ministry
of Education, Beijing 100875, People's Republic of China

\smallskip

\noindent{\it E-mail:} \texttt{dcyang@bnu.edu.cn}

\bigskip

\noindent Zunwei Fu

\medskip

\noindent Department of Mathematics, Linyi University, Linyi 276005, People's Republic of China

\smallskip

\noindent{\it E-mail:} \texttt{zwfu@mail.bnu.edu.cn}

\end{document}